\documentclass{amsart}
\usepackage{amssymb,amsmath, amsfonts, amsrefs, tikz, epsfig, float}
\usepackage{amsthm}
\usepackage{mathrsfs}
\usepackage{enumitem, indentfirst}
\usepackage[fleqn,tbtags]{mathtools}
\usepackage{color}
 \newtheorem{theorem}{Theorem}
  \newtheorem*{maintheorem}{Theorem}
  \newtheorem*{dougtheorem}{Douglas' Factorization Theorem}

 \newtheorem{claim}[theorem]{Claim}

 \theoremstyle{definition}

 \theoremstyle{remark}

 \newtheorem*{remark}{Remark}
  \numberwithin{equation}{section}

\renewcommand{\theta}{\vartheta}

\DeclareMathOperator{\tform}{\mathfrak{t}}

\DeclareMathOperator{\wform}{\mathfrak{w}}

\DeclareMathOperator{\ran}{ran}

\DeclareMathOperator{\lath}{Lat(\hil)}
\DeclareMathOperator{\lategyh}{Lat_1(\hil)}

\DeclareMathOperator{\glh}{GL(\hil)}

\DeclarePairedDelimiterX\sipt[2]{(}{)_{\tform}}{#1\,\delimsize\vert\,#2}
\DeclarePairedDelimiterX\sipv[2]{(}{)_{v}}{#1\,\delimsize\vert\,#2}
\DeclarePairedDelimiterX\sipw[2]{(}{)_{w}}{#1\,\delimsize\vert\,#2}

\newcommand{\dupN}{\mathbb{N}}

\newcommand{\seq}[1]{(#1_{n})_{n\in\dupN}}

\newcommand{\dupC}{\mathbb{C}}

\newcommand{\hil}{\mathcal{H}}

\newcommand{\bh}{\mathcal{B}(\hil)}

\DeclarePairedDelimiterX\abs[1]{\lvert}{\rvert}{#1}
\DeclarePairedDelimiterX\sip[2]{\langle}{\rangle}{#1\,,#2}
\DeclarePairedDelimiterX\siptilde[2]{(}{)_{\!_{\widetilde{A}}}}{#1\,\delimsize\vert\,#2}
\DeclarePairedDelimiterX\sipf[2]{(}{)_{f}}{#1\,\delimsize\vert\,#2}
\DeclarePairedDelimiterX\sipg[2]{(}{)_{g}}{#1\,\delimsize\vert\,#2}
\DeclarePairedDelimiterX\siptw[2]{(}{)_{\tform+\wform}}{#1\,\delimsize\vert\,#2}
\DeclarePairedDelimiterX\set[2]{\{}{\}}{#1\,:\,#2}
\DeclarePairedDelimiterX\dual[2]{\langle}{\rangle}{#1,#2}
\DeclarePairedDelimiterX\sipa[2]{(}{)_{\!_A}}{#1\,\delimsize\vert\,#2}
\DeclarePairedDelimiterX\sipc[2]{(}{)_{\!_C}}{#1\,\delimsize\vert\,#2}
\DeclarePairedDelimiterX\sipab[2]{(}{)_{\!_{A+B}}}{#1\,\delimsize\vert\,#2}
\DeclarePairedDelimiterX\sipb[2]{(}{)_{\!_B}}{#1\,\delimsize\vert\,#2}

\allowdisplaybreaks
\begin{document}
\title[Maps preserving Douglas solution]{Maps preserving the Douglas solution of operator equations}

\author[Zs. Tarcsay]{Zsigmond Tarcsay}
\thanks{The corresponding author Zs. Tarcsay was supported by DAAD-TEMPUS Cooperation Project ``Harmonic Analysis and Extremal Problems'' (grant no. 308015),  by the J\'anos Bolyai Research Scholarship of the Hungarian Academy of Sciences, and by the \'UNKP--20-5-ELTE-185 New National Excellence Program of the Ministry for Innovation and Technology. ``Application Domain Specific Highly Reliable IT Solutions'' project  has been implemented with the support provided from the National Research, Development and Innovation Fund of Hungary, financed under the Thematic Excellence Programme TKP2020-NKA-06 (National Challenges Subprogramme) funding scheme.}
\address{%
Zs. Tarcsay \\ Department of Applied Analysis  and Computational Mathematics\\ E\"otv\"os Lor\'and University\\ P\'azm\'any P\'eter s\'et\'any 1/c.\\ Budapest H-1117\\ Hungary\\ and 
Alfr\'ed R\'enyi Institute of Mathematics\\ Re\'altanoda utca 13-15.\\ Budapest H-1053\\ Hungary}
\email{tarcsay@cs.elte.hu}



\begin{abstract}
We consider bijective maps $\phi$ on the full operator algebra $\bh$ of an infinite dimensional Hilbert space with the property that, for every $A,B,X\in\bh$, $X$ is the Douglas solution of the equation $A=BX$  if and only if  $Y=\phi(X)$ is the Douglas solution of the equation $\phi(A)=\phi(B)Y$. We prove that those maps are implemented by a unitary or anti-unitary map $U$, i.e., $\phi(A)=UAU^*$.    
\end{abstract}

\maketitle

\section{Introduction}

 Operator equations of the form 
\begin{equation}\label{E:BX=A}
    BX=A
\end{equation}  
arise in many problems in engineering, physics, and statistics. In \cite{douglas} R. G. Douglas considered the problem in the context of Hilbert space operators. He established the following characterisation of solvability: 
\begin{dougtheorem}
Let $\hil$ be a  Hilbert space and let $A,B\in\bh$. Then the following statements are equivalent:
    \begin{enumerate}[label=\textup{(\roman*)}]
        \item the operator equation $BX=A$ has a solution $X\in\bh$,
        \item  there exists $\lambda>0$ s.t. $\|A^*x\|\leq \lambda \|B^*x\|$ for all $x\in\hil$, 
        \item  we have the range inclusion $\ran A\subseteq \ran B$. 
    \end{enumerate}
    In any case, there is a unique operator $X=D$ such that 
    \begin{enumerate}
        \item[$(D1)$] $\ran D\subseteq [\ker B]^\perp$,
        \item[$(D2)$] $\ker D=\ker A$,
        \item[$(D3)$] $\|D\|=\inf\{\lambda>0\,:\,\|A^*x\|\leq \lambda \|B^*x\|\quad (\forall x\in\hil) \}$.
    \end{enumerate}
\end{dougtheorem}
In fact, condition $(D1)$ automatically implies both $(D2)$ and $(D3)$. The unique operator $D$ satisfying conditions $(D1)$-$(D3)$ is called the \textit{Douglas solution}  (or  \textit{reduced solution}) of \eqref{E:BX=A}. 

There are many important objects in operator theory, including the Moore-Penrose pseudoinverse \cite{Corach}, the parallel sum \cite{F-W} or the Schur complement \cite{TZS-Schur}, which can be defined as the Douglas solution of a suitably posed operator equation. Generalisations of Douglas' factorisation theorem to Banach spaces \cites{Embry, Barnes}, locally convex spaces \cite{TZs-jmaa}, Hilbert $C^*$-modules \cite{Moslehian} and unbounded operators and relations \cites{Forough, PopSeb} are also available.

In this paper we consider the following nonlinear preserver problem: Let $\phi:\bh\to\bh$ be a bijective map with the property that, for every triple $A,B, X$ of bounded operators in $\bh$,  $X$ is the Douglas solution of the equation $A=BX$ if and only if $Y=\phi(X)$ is the Douglas solution of the equation $\phi(A)=\phi(B)Y$. (Shortly, we say in that case that  $\phi$ preserves the Douglas solution in both directions.) In this note we describe the form all such transformation $\phi$. Our result shows that the structure of those mappings is quite rigid, namely, every Douglas solution preserving map $\phi$ is of the form
\begin{equation*}
    \phi(A)= UAU^*,\qquad A\in\bh,
\end{equation*}
for a fixed unitary or anti-unitary map $U$. 


We stress that the only constraint concerns the dimension of the underlying Hilbert space, and no algebraic or topological assumptions like linearity or continuity on $\phi$ are imposed. On the contrary: these properties follow from the intrinsic structure of such a transformation.
\section{The main theorem}

Before we state and prove our main result, let us fix some notation. Throughout, $\hil$ denotes an infinite dimensional complex Hilbert space, and $\bh$ stands for the $C^*$-algebra of all bounded, linear operators $A:\hil\to\hil$. The kernel and range spaces of an operator $A\in\bh$ are denoted by $\ker A$ and $\ran A$, respectively. The collection of the ranges of all bounded operators is denoted by $\lath$, that is, 
\begin{equation*}
  \lath\coloneqq \set{\ran A}{A\in\bh}.
\end{equation*}
Note that $\lath$ is not identical with the class of all linear submanifolds of $\hil$ (cf. \cite{F-W}). For the set of one-dimensional subspaces of $\hil$ we use the symbol $\lategyh$:
\begin{equation*}
    \lategyh\coloneqq \set{\dupC e}{e\in\hil,e\neq0}.
\end{equation*}
For given two vectors $e,f\in\hil$, we define the one-rank operator $e\otimes f$ by
\begin{equation*}
    (e\otimes f)(x)\coloneqq \sip{x}{f} e,\qquad x\in\hil.
\end{equation*}
If $\mathcal M\in\lath$ is a closed subspace of $\hil$ then we denote by $P_{\mathcal M}$ the orthogonal projection onto $\mathcal M$. If $\mathcal M$ is of the form $\mathcal M=\dupC e$ for  some $e\in\hil$ then we set $P_e\coloneqq P_{\mathcal M}$. Finally, the set of invertible operators (that is, of those operators which do not have $0$ in their spectrum) is denoted by $\glh$.

\begin{maintheorem}\label{T:maintheorem}
Let $\hil$ be an infinite dimensional Hilbert space. A bijective map $\phi:\bh\to\bh$  preserves the Douglas solution in both directions if and only if there exists a unitary or anti-unitary operator $U:\hil\to\hil$ such that 
\begin{equation}\label{E:UAU}
    \phi(A)=UAU^*,\qquad  A\in\bh.
\end{equation}
\end{maintheorem}
\begin{proof} We start by observing that the bijective map $\phi:\bh\to\bh$ preserves the Douglas solution in both directions if and only if 
\begin{align}\label{E:iff}
        \left.\begin{array}{cc}
             A=BX  \\
             \ran X\subseteq [\ker B]^{\perp} 
        \end{array}\right\}\iff 
        \left\{\begin{array}{cc}
             \phi(A)=\phi(B)\phi(X)  \\
             \ran \phi(X)\subseteq [\ker \phi(B)]^{\perp} 
        \end{array}\right.
\end{align}
for all $A,B,X\in\bh$. From this we infer that $\phi$ satisfies
\begin{equation}\label{E:BD1}
    \phi(BD)=\phi(B)\phi(D),\qquad \forall B,D\in\bh, \ran D\subseteq [\ker B]^\perp.
\end{equation}

\begin{claim}\label{Cl:surjective} 
$\ran(B) = \mathcal{H} \iff \ran(\phi(B))= \mathcal{H}, \qquad (B \in \mathcal{B}(\mathcal{H}))$
\end{claim}

Indeed, if $\ran B=\hil$ then equation \eqref{E:BX=A} is solvable for every $A\in\bh$, and therefore the same is true for equation  $\phi(B)Y=\phi(A)$. Since $\phi$ is bijective, this implies $\ran\phi(B)=\hil$. The converse direction is proved similarly. 

\begin{claim}\label{Cl:injective}
$\ker B=\{0\}\iff\ker\phi(B)=\{0\},\qquad (B \in \mathcal{B}(\mathcal{H}))$.
\end{claim}

For if $\ker B=\{0\}$, then the Douglas solution of $BX=B$ is $X=I$, so by \eqref{E:iff} we get
\begin{equation*}
    \phi(B)=\phi(B)\phi(I), \quad \mbox{and}\quad \ran \phi(I)\subseteq [\ker\phi(B)]^\perp.
\end{equation*}
Since $\phi(I)$ is surjective, we infer that $\ker \phi(B)=\{0\}$. This proves Claim \ref{Cl:injective}. 

We see from Claims 1 \& 2 and equation \eqref{E:BD1} that the restriction $\phi |_{\glh}$ of $\phi$ onto the group of invertible elements $\glh$ is in fact an automorphism of $\glh$.
Hence, by \cite{AutGLH}*{Theorem 3.1},  there exists a linear or conjugate linear topological isomorphism $S:\hil\to\hil$ such that either
\begin{equation}\label{E:oke}
    \phi(X)=SXS^{-1},\qquad \forall X\in\glh,
\end{equation}
or 
\begin{equation}\label{E:lehetetlen}
    \phi(X)=(SX^{-1}S^{-1})^*,\qquad \forall X\in\glh.
\end{equation}
In Claim \ref{Cl:4} we shall demonstrate that \eqref{E:lehetetlen} may not happen. To do so we first need the following.
\begin{claim}\label{Cl:3}
$\phi$ 
maps orthogonal projections into orthogonal projections.
\end{claim}
 Indeed, if $P=P^2=P^*$ then $\ran P =[\ker P]^{\perp}$ and therefore \eqref{E:BD1} implies 
\begin{equation*}
    \phi(P)=\phi(P\cdot P)=\phi(P)\phi(P),
\end{equation*}
hence $\phi(P)$ is an idempotent. Furthermore, the Douglas solution of the equation  $P=PX$ is just $X=P$, hence the Douglas solution of $\phi(P)=\phi(P)Y$ is $Y=\phi(P)$, which in turn implies $\ran\phi(P)\subseteq [\ker\phi(P)]^\perp$. Thus $\phi(P)$ is an orthogonal projection.  

\begin{claim}\label{Cl:4}
\eqref{E:lehetetlen} cannot hold. 
\end{claim}

 For assume towards a contradiction that it does. Let $\seq e$ be an orthonormal sequence in $\hil$, and  take   a self-adjoint operator $A\in\bh$  such that $Ae_n=\frac1n e_n$. Let us denote by $P_n\coloneqq e_n\otimes e_n$ the orthogonal projection onto $\dupC e_n$. Then we have $AP_n=\frac1n P_n$ and $\ran P_n\subseteq [\ker A]^{\perp}$. From \eqref{E:BD1} and \eqref{E:lehetetlen} we thus obtain 
\begin{equation*}
    \phi(A)\phi(P_n)=\phi(AP_n)=\phi(\tfrac1n IP_n)=\phi(\tfrac 1n I)\phi(P_n)=n\phi(P_n).
\end{equation*}
Since $\phi(P_n)$ is a non-zero orthogonal projection, this means that every integer $n$ is an eigenvalue of $\phi(A)$, that is impossible. This proves Claim \ref{Cl:4}.

\bigskip

So we see now that $\phi$ acts on $\glh$ by $\eqref{E:oke}$. Next we observe that $\phi$ preserves range inclusion in both directions, i.e., 
\begin{equation*}
    \ran A\subseteq \ran B\quad \iff\quad \ran \phi(A)\subseteq \ran \phi(B),\qquad A,B\in\bh,
\end{equation*}
which is obvious by the very definition of Douglas solution preserving maps. Hence $\phi$ induces a bijective map $\Phi:\lath\to\lath$ on the set of  operator ranges by the correspondence
\begin{equation*}
    \Phi(\ran A)\coloneqq \ran \phi(A),\qquad A\in\bh.
\end{equation*}
\begin{claim}\label{Cl:5}
The restriction $\Phi |_{\lategyh}$ of $\Phi$ onto $\lategyh$ is a bijection of $\lategyh$.
\end{claim}

Let $B\in\bh$ be a rank one operator and take two non-zero vectors $e_1,e_2\in \ran \Phi(B)$. Let $E_1,E_2\in\bh$ be rank one operators such that $\ran E_i=e_i$, then the equations $\Phi(B)Y=E_i$, $i=1,2$, are both solvable and therefore, setting $C_i\coloneqq \phi^{-1}(E_i)$,  equations $BX=C_i$ are solvable too. Since $B$ has rank one, it follows that $\ran C_1=\ran C_2=\ran B$ and therefore $\ran E_1=\ran E_2$, hence $e_1$ and $e_2$ are linearly dependent. A similar argument applied to $\phi^{-1}$ shows that $\Phi:\lategyh \to \lategyh $ is a bijection. This proves Claim \ref{Cl:5}. 

\bigskip
It is clear that for all non-zero vectors $e,f,g\in\hil$  one has
\begin{equation*}
    \dupC g\subseteq \dupC e+\dupC f\quad \iff\quad \Phi(\dupC g)\subseteq \Phi(\dupC e)+\Phi(\dupC f),
\end{equation*}
hence $\Phi$ is a projectivity in the sense of \cite{Faure2}. The fundamental theorem of projective geometry (see e.g. \cite{Faure2}) implies that $\Phi$ is implemented by a semi-linear bijection $T:\hil\to\hil$  in the sense that  
\begin{equation*}
    \Phi(\dupC e)=\dupC \cdot Te,\qquad e\in\hil, e\neq 0.
\end{equation*}
Moreover, $T$ is unique up to a scalar multiplier.
Take now an $\mathcal{M}\in\lath$, then  
\begin{equation*}
    \Phi(\mathcal{M})=\Phi\bigg(\bigcup_{e\in\mathcal M} \dupC e\bigg)=\bigcup_{e\in\mathcal M}\Phi(\dupC e)=\bigcup_{e\in\mathcal M}\dupC\cdot Te=T(\mathcal M),
\end{equation*}
whence it follows that
\begin{equation}\label{E:ranTA}
    \ran \phi(A)=\ran TA,\qquad A\in\bh.
\end{equation}
By \cite{F-L}*{Theorem 1} we infer that $T:\hil\to\hil$ is either a linear or a conjugate linear topological isomorphism. 

\begin{claim}\label{Cl:S=lambdaT}
We claim that $S= \lambda T$ for some non-zero scalar $\lambda$.
\end{claim}
Let us assume first that both $S,T$ are  linear. In this case, $\phi(\lambda A)=\lambda \phi(A)$ for every $A\in\bh$ and $\lambda\in\dupC$. By \eqref{E:ranTA}, 
\begin{equation*}
    \phi(P_x)=P_{Tx},\qquad x\in\hil,
\end{equation*}
where $P_x$ denotes the orthogonal projection onto $\dupC x$. Hence 
\begin{equation*}
    \phi(XP_x)=\phi(X)\phi(P_x)=SXS^{-1}P_{Tx},\qquad X\in\glh, x\in\hil. 
\end{equation*}
Consequently, for every fixed $X\in\glh$,
\begin{equation*}
    \dupC\cdot TXx=\ran (TXP_x)=\ran \phi(XP_x)=\dupC\cdot SXS^{-1}Tx, \qquad \forall x\in\hil.
\end{equation*}
The uniqueness part of the fundamental theorem of projective geometry implies 
\begin{equation*}
    TX=\lambda_XSXS^{-1}T,\qquad X\in\glh,
\end{equation*}
where $\lambda_X\in \dupC$ is a scalar depending on $X$. After rearranging this yields $S^{-1}TX=\lambda_XXS^{-1}T$, whence, by setting $R\coloneqq S^{-1}T$, we obtain that 
\begin{equation}\label{E:RX}
    RX=\lambda_XXR,\qquad \forall X\in \glh.
\end{equation}
It is easy to check that the map $X\mapsto \lambda_X$ is continuous with $\lambda_I=1$, $\lambda_{XY}=\lambda_X\lambda_Y$ and $\lambda_{\alpha X}=\lambda_X$, $(0\neq \alpha\in\dupC)$. If $X,Y,X+Y\in\glh$ then also 
\begin{equation*}
    \lambda_{X+Y}(X+Y)=\lambda_XX+\lambda_YY,
\end{equation*}
whence we obtain 
\begin{equation*}
    [\lambda_{X+Y}-\lambda_X]I=[\lambda_Y-\lambda_{X+Y}]YX^{-1}.
\end{equation*}
Therefore, if $YX^{-1}\notin\dupC I$, then $\lambda(X)=\lambda(X+Y)=\lambda(Y)$. For every $X\in\glh$ with $\abs{\alpha}<\|X\|^{-1}$ we have $I+\alpha X\in\glh$ and so
\begin{equation*}
    \lambda_X=\lambda_{\alpha X}=\lambda _{I+\alpha X}=\lambda_I=1.
\end{equation*}
According to \eqref{E:RX} this means that $R$ commutes with every $X\in\glh$ and therefore $R\in\dupC I$. Consequently, $S=\alpha T$ for some (non-zero) $\alpha$.  

The same argument applies when $S$ and $T$ are both conjugate linear. 

Assume now that $S$ is linear and $T$ is conjugate linear, or conversely, that $S$ is conjugate linear and $T$ is linear. Just like above we conclude that \eqref{E:RX} holds true with the conjugate linear operator $R=S^{-1}T$. Then we have the following identities on the spectra:
\begin{equation*}
    \sigma(X^*)=\sigma(RXR^{-1})=\lambda_X\cdot \sigma(X),\qquad \forall X\in\glh
\end{equation*} 
But this cannot hold for every $X\in \glh$ (take e.g. an $X$ with spectrum $\sigma(X)=\{1,i,-2i\}$), so this is a contradiction. The proof of Claim \ref{Cl:S=lambdaT} is therefore complete.

\bigskip

We may and we will therefore assume that $S=T$.

\begin{claim}\label{Cl:7}
$T$ is unitary or anti-unitary
\end{claim}

Let  $P$ be an orthogonal projection, then $2P-I$ is unitary such that $(2P-I)P=P$. Hence 
\begin{equation*}
    \phi(P)=\phi(2P-I)\phi(P)=(2TPT^{-1}-I)\phi(P),
\end{equation*}
whence we get 
\begin{equation}\label{E:}
    \phi(P)=TPT^{-1}.
\end{equation}
In particular, for $e\in\hil, \|e\|=1$,  we have 
\begin{equation}
    \phi(e\otimes e)=Te\otimes (T^{-1})^*e.
\end{equation}
Furthermore, 
\begin{equation*}
    \phi(e\otimes e)=\phi(P_e)=P_{Te}=\frac{1}{\|Te\|^2}Te\otimes Te,
\end{equation*}
whence  $(T^{-1})^*e=\frac{1}{\|Te\|^2}Te$, and therefore   
\begin{equation*}
     Tx=\frac{\|Tx\|^2}{\|x\|^2}(T^{-1})^*x,\qquad x\in\hil, x\neq 0.
\end{equation*}
This yields that $ x\perp y$ implies $Tx\perp Ty$ for every pair of non-zero vectors  $x,y\in\hil$. An easy calculation shows that $\|Te\|^2=\|Tf\|^2\eqqcolon\alpha$ for every unit vectors $e,f\in\hil$, $e\perp f$.  It follows therefore that $U\coloneqq \alpha^{-1} T$ is either a unitary or an anti-unitary operator. 
\bigskip

We may therefore assume that $U=T$. Then 
\begin{equation}\label{E:phieoe}
    \phi(e\otimes e)=Ue\otimes Ue,\qquad e\in \hil, \|e\|=1.
\end{equation}
Let now $e,f\in \hil$ be unit vectors, then $\phi(e\otimes f)$ is of   rank one, namely, $\ran\phi(e\otimes f)=\dupC \cdot Ue$. Consequently,  $\phi(e\otimes f)=Ue\otimes u_{e,f}$ for some $u_{e,f}\in\hil$ (depending on the vectors $e,f$). Since  we have
\begin{equation*}
    e\otimes e=(e\otimes f)(f\otimes e), \qquad \ran (f\otimes e)\subseteq [\ker(e\otimes f)]^\perp,
\end{equation*}
from \eqref{E:BD1} it follows that 
\begin{align*}
    Ue\otimes Ue=(Ue\otimes u_{e,f})(Uf\otimes u_{f,e})=\sip{Uf}{u_{e,f}}\cdot (Ue\otimes u_{f,e}), 
\end{align*}
consequently $u_{f,e}\in\dupC\cdot Ue$. In particular, 
\begin{equation*}
    \phi(e\otimes f)=\lambda_{e,f} \cdot Ue\otimes Uf,
\end{equation*}
for some $\lambda_{e,f}\in\dupC$. Consider now a unitary operator $V\in\bh$ such that $Ve=f$. Then $f\otimes f=V(e\otimes f)$, whence it follows that 
\begin{equation*}
    Uf\otimes Uf=\phi(V)\phi(e\otimes f)=UVU^{-1}(\lambda_{e,f}Ue\otimes Uf)= \lambda_{e,f}\cdot Uf\otimes Uf,
\end{equation*}
and therefore $\lambda_{e,f}=1$. Consequently, $\phi(e\otimes f)=Ue\otimes Uf$ for every pair of unit vectors $e,f\in\hil$. This implies  
\begin{equation}\label{E:phieof2}
    \phi(x\otimes y)=Ux\otimes Uy,\qquad x,y\in \hil.
\end{equation}

To conclude the proof, consider an arbitrary operator $A\in\bh$. Take first a unit vector $e\in[\ker A]^\perp$ then the identities $\phi(Ae\otimes e)=\phi(A)\otimes \phi(e\otimes e)$ and \eqref{E:phieof2} imply
\begin{equation*}
    UAe\otimes Ue=\phi(A)Ue\otimes Ue. 
\end{equation*}
Hence  $UAe=\phi(A)Ue$. This proves that
\begin{equation}\label{E:phiAUe1}
    \phi(A)Ux=UAx,\qquad \forall x\in[\ker A]^\perp.
\end{equation}
Denote  by $P$ the orthogonal projection onto $[\ker A]^\perp$. Since the Douglas solution of $A=AX$ is $X=P$, the Douglas solution of $\phi(A)=\phi(A)Y$ is in turn $Y=\phi(P)$. Here,  $\phi(P)$ coincides with the orthogonal projection onto $U(\ker A^{\perp})$. Thus we conclude that
\begin{equation*}
    \ker \phi(A)=\ker \phi(P)=[U(\ker A^{\perp})]^\perp=U(\ker A).
\end{equation*}
In particular, 
\begin{equation}\label{E:phiAUx2}
    0=\phi(A)Ux=UAx,\qquad \forall x\in\ker A.
\end{equation}
From  identities \eqref{E:phiAUe1} and \eqref{E:phiAUx2} it follows that $\phi$ satisfies \eqref{E:UAU}.
\end{proof}
\begin{remark}
The above proof relies on tools which depend on the infinite-dimensionality of $\hil$.  It remains an intriguing open question whether the statement of the Theorem remains true in the finite dimensional case. 

\end{remark}
\textbf{Acknowledgement.} The author is grateful to Gy\"orgy P. Geh\'er for many enlightening conversations about the contents of the paper. He is also extremely grateful to Bence Horv\'ath for carefully reading the paper and for his valuable comments which much improved the exposition of the paper.

\end{document}